\documentclass[12pt,reqno]{amsart}

\usepackage{verbatim}
\usepackage{amssymb}
\usepackage[dvipsnames]{xcolor}
\usepackage[pdftex]{hyperref}

\hypersetup{colorlinks=true, linkcolor=blue, citecolor=ForestGreen}

\newcommand{\cA}{{\mathcal A}}

\newcommand{\longc}{,\dotsc,}
\newcommand{\longp}{+\dotsb+}
\newcommand{\longu}{\cup\dotsb\cup}

\newcommand{\seq}{\subseteq}
\newcommand{\stm}{\setminus}

\renewcommand{\phi}{\varphi}
\newcommand{\sig}{\sigma}

\renewcommand{\Im}{\mathrm{Im}}

\renewcommand{\>}{\rangle}

\newcommand{\lfl}{\left\lfloor}
\newcommand{\rfl}{\right\rfloor}

\newcommand{\lpr}{\left(}
\newcommand{\rpr}{\right)}

\newcommand{\lcl}{\left\lceil}
\newcommand{\rcl}{\right\rceil}

\newtheorem{lemma}{Lemma}
\newtheorem{theorem}{Theorem}
\newtheorem{conjecture}{Conjecture}
\newtheorem{proposition}{Proposition}

\newcommand{\refj}[1]{\ref{j:#1}}
\newcommand{\refl}[1]{\ref{l:#1}}
\newcommand{\refp}[1]{\ref{p:#1}}
\newcommand{\reft}[1]{\ref{t:#1}}
\newcommand{\refs}[1]{\ref{s:#1}}
\newcommand{\refb}[1]{\cite{b:#1}}
\newcommand{\refe}[1]{\eqref{e:#1}}

\title%
[Quotient sets in nonabelian groups]{Quotient sets \\ in nonabelian groups}
\author{Vsevolod F. Lev}
\email{seva@math.haifa.ac.il}
\address{Department of Mathematics, The University of Haifa at Oranim,
	Tivon 36006, Israel}
\makeatletter
\@namedef{subjclassname@2020}{\textup{2020} Mathematics Subject Classification}
\makeatother
\subjclass[2020]{Primary 11B75}
\keywords{Small doubling, Sumset, Kneser's theorem}

\begin{document}
\baselineskip=16pt

\begin{abstract}
We show that for a finite, nonempty subset $A$ of a  group, the quotient  
set $A^{-1}A:=\{a_1^{-1}a_2\colon a_1,a_2\in A\}$ has size 
$|A^{-1}A|\ge\frac53\,|A|$, unless $A$ is densely contained in a coset, or 
in a union of two cosets of a finite subgroup. 
\end{abstract}

\maketitle

\section{Introduction: Background and the Main Result}

One of the cornerstones of the additive combinatorics is Kneser's 
theorem~\cite{b:kn1,b:kn2} relating the size of a sumset in an abelian group 
to the sizes of the set summands. As shown by Olson~\refb{o1}, a 
straightforward, simple-minded analogue of Kneser's theorem for nonabelian 
groups fails to hold. Many partial extensions of the theorem in the 
nonabelian settings are known, however; see, for instance, 
\cite{b:f,b:o1,b:o2,b:sw,b:t,b:z,b:h}. Particularly relevant in our present 
context are the papers by Freiman~\refb{f}, Olson~\refb{o2}, and 
Hamidoune~\refb{h}. 

In~\refb{f}, Freiman classified finite, nonempty subsets $A$ of a group with 
the product set $A^2:=\{a_1a_2\colon a_1,a_2\in A\}$ satisfying 
$|A|^2<\frac85\,|A|$. 

Olson has extended Freiman's result onto products with distinct set factors; 
namely, as shown in~\cite[Theorem~1]{b:o1}, if $A$ and $B$ are finite, 
nonempty subsets of a group, then ``normally'' the product set 
$AB:=\{ab\colon a\in A,b\in B\}$ has size $|AB|\ge|A|+\frac12\,|B|$. 

Improving the \emph{doubling coefficients} $\frac85$ and $\frac12$ in the   
Freiman-Olson estimates is a fascinating, mostly open, problem.

Addressing the case where $B=A^{-1}$, Hamidoune~\refb{h} has established some 
properties of the quotient set  $A^{-1}A:=\{a_1^{-1}a_2\colon a_1,a_2\in A\}$ 
assuming that $|A^{-1}A|<\frac53\,|A|$. 

In this note, under the same assumption $|A^{-1}A|<\frac53\,|A|$, we  
completely determine the structure of the set $A$ itself, with an 
if-and-only-if-type classification. 

For a subgroup $H$ of a group $G$, let $N(H)$ denote the  normalizer of $H$ 
in $G$. 
\begin{theorem}\label{t:main}
Let $A$ be a finite subset of a group $G$. Then $|A^{-1}A|<\frac53\,|A|$ if 
and only if one of the following holds: 
\begin{itemize}
\item[(i)] there is a finite subgroup $H\le G$ such that $A$ is contained 
    in a left $H$-coset and $|A|>\frac35\,|H|$;  
\item[(ii)] there is a finite subgroup $H\le G$ and elements $a,b\in G$ 
    with $(a^{-1}b)^2\notin H$ and $a^{-1}b\in N(H)$ such that $A\seq 
    aH\cup bH$ and  $|A|>\frac95\,|H|$. 
\end{itemize}
Moreover, in the case (i) we have $A^{-1}A=H$, while in the case (ii) the set 
$A^{-1}A$ is a disjoint union of $H$ and two double $H$-cosets of size $|H|$ 
each.  
\end{theorem}

The coefficient $\frac53$ corresponds to a structure threshold: say, if $H$ 
is a finite subgroup, and $g\in N(H)$ is an element with $g^i\notin H$ for 
$i\in[1,4]$, 
then the set $A:=g^{-1}H\cup 
H\cup Hg$ satisfies $|A^{-1}A|=\frac 53\,|A|$, while $A$ does not have the 
structure described in the theorem. 

We remark that the ostensible lack of symmetry in the statement of the 
theorem clears off once we notice that any left coset is a right coset of a 
conjugate subgroup, and vice versa. 

With the exception of Section~\refs{conjecture}, the rest of the paper is 
devoted to the proof of Theorem~\reft{main}. In the next section we formally 
introduce the notation used and gather some basic facts needed for the proof. 
In Section~\refs{suff} we show that conditions (i) and (ii) of 
Theorem~\reft{main} are sufficient for $|A^{-1}A|<\frac53\,|A|$ to hold, and 
also that they imply the last assertion of the theorem (concerning the 
structure of the quotient set). Additionally, in Section~\refs{suff} we prove 
a lemma that will be used in the course of the proof of necessity in 
Section~\refs{proof}. Finally, in the concluding Section~\refs{conjecture} we 
state and briefly discuss a conjectural extension of Theorem~\reft{main} onto 
the sets $A$ satisfying $|A^{-1}A|<2|A|$.  
 
\section{Preliminaries: notation and tools}\label{s:prelim}

For subsets $A$ and $B$ of a group, we denote by $A^{-1}$ the set of  
inverses of the elements of $A$, and by $AB$ the product set: 
  $$ A^{-1}:=\{a^{-1}\colon a\in A \}
       \quad \text{and} \quad AB:=\{ab\colon a\in A,b\in B\}. $$
Thus, for instance, $A^{-1}A=\{a^{-1}b\colon a,b\in A\}$.

The subgroup generated by $A$ is denoted by $\<A\>$, and the identity element 
of the group by $1$. A left (right) coset of a subgroup $H$ is a set of the 
form $gH$ ($Hg$), where $g$ is an element of the group. A \emph{double} 
$H$-coset if a set of the form $HgH$. 

The following lemma summarizes the basic properties of double cosets.
\begin{lemma}
If $H$ is a subgroup of a group $G$, then $G$ is a disjoint union of double 
$H$-cosets. A set $S\seq G$ is a union of double $H$-cosets if and only if it 
is stable under both left and right multiplication by $H$; that is, if and 
only if $HS=SH=S$; alternatively, if and only if there is a set $T\seq G$ 
such that $S=HTH$. For $a,b\in G$, we have $HaH=HbH$ if and only if there 
exist $h_1,h_2\in H$ such that $b=h_1ah_2$. 
\end{lemma}

The \emph{normalizer} of a subgroup $H$, denoted $N(H)$, is the subgroup 
consisting of all those group elements $g$ satisfying $gH=Hg$. 

Given a finite subset $A$ of a group, and a group element $g$, by $r(g)$ we 
denote the number of representations of $g$ in the form $g=a^{-1}b$ with 
$a,b\in A$. Clearly, $r$ is supported on the quotient set $A^{-1}A$, and 
$r(g^{-1})=r(g)\le r(1)=|A|$ for any element $g$. Moreover, $r(g)=|A|$ if and 
only if $Ag=A$. Therefore, the number of elements $g\in G$ satisfying $Ag=A$ 
is the size of the maximal subgroup $H$ such that $A$ is a union of left 
$H$-cosets. 

For a real $x$, the largest integer not exceeding $x$ and the smallest 
integer not smaller than $x$ are denoted $\lfl x\rfl$ and $\lcl x\rcl$, 
respectively. 

\begin{lemma}\label{l:leftright}
For any finite subgroup $H$ and any group elements $a$ and $b$, either  
$aH=Hb$, or $|aH\cap Hb|\le\frac12\,|H|$. 
\end{lemma}

\begin{proof}
Assuming that $aH\cap Hb$ is nonempty, fix an element $g\in aH\cap Hb$. Then 
$a\in gH$ and $b\in Hg$, whence $|aH\cap Hb|=|gH\cap Hg|=|H\cap g^{-1}Hg|$. 
The result follows since the intersection in the right-hand side is a 
subgroup of $H$. 
\end{proof}

\begin{lemma}\label{l:sub}
Suppose that $H$ is a finite subgroup, and $a,b$ are elements of a group.  
For $aH=Hb$ to hold, it is necessary and sufficient that $a,b\in N(H)$ and 
$aH=bH$. 
\end{lemma}

\begin{proof}
If $aH=Hb$, then $b\in aH$ whence $a^{-1}b\in H$; equivalently, $b^{-1}a\in 
H$, or $a\in bH$. As a result, $aH=bH$, and, consequently, $Hb=bH$, so that 
$b\in N(H)$. In a similar way we get $a\in N(H)$. 	The opposite direction is 
trivial: if $a,b\in N(H)$ and $aH=bH$, then $aH=bH=Hb$. 
\end{proof}

\begin{lemma}\label{l:H2}
Suppose that $H$ is a subgroup, and $g\notin H$ is an element of a group. For 
the union $H\cup gH$ to be a subgroup, it is necessary and sufficient that 
$g^2\in H$. 
\end{lemma}

\begin{proof}
If $H\cup gH$ is a subgroup, then $g^{-1}\in H\cup gH$ whence, indeed, 
$g^{-1}\in gH$, and then $g^2\in H$. Conversely, if $g^2\in H$, then $H\cup 
gH$ is easily seen to be closed under the ``skew multiplication'' 
$(a,b)\mapsto a^{-1}b$. 
\end{proof}

\begin{lemma}
Suppose that $H$ is a finite subgroup of a group $G$. For an element $g\in 
G$, the double $H$-coset $HgH$ has size $|HgH|=|H|$ if and only if $g\in  
N(H)$. 
\end{lemma}

\begin{proof}
Write $S=HgH$. Then $gH\seq S$ and $Hg\seq S$. Hence, $|S|=|H|$ if and only 
if $gH=Hg$; that is, if and only if $g\in N(H)$. 
\end{proof}

We will use the box principle in the following form. 
\begin{lemma}\label{l:box}
Suppose that $A$ is a finite, nonempty subset of a group $G$. If $g_1,g_2\in 
G$ are group elements with $r(g_1)+r(g_2)>|A|$, then $g_1^{-1}g_2\in 
A^{-1}A$. 
\end{lemma}

\begin{proof}
For $i\in\{1,2\}$, let $A_i$ be the set of all those elements $a\in A$ the 
inverse of which appears as the first factor in some representation 
$g_i=a^{-1}b$ with $b\in A$. Thus, $|A_i|=r(g_i)$, and from 
$r(g_1)+r(g_2)>|A|$ and $A_1,A_2\seq A$ it follows that $A_1$ and $A_2$ have 
a common element; that is, there are $a,b_1,b_2\in A$ such that 
$g_1=a^{-1}b_1$ and $g_2=a^{-1}b_2$. Consequently, 
$g_1^{-1}g_2=b_1^{-1}aa^{-1}b_2=b_1^{-1}b_2\in A^{-1}A$. 
\end{proof}

We need the following result of Kemperman and Wehn. 


\begin{theorem}[Kemperman-Wehn]\label{t:KW}
If $A$ and $B$ are finite, nonempty subsets of a group, then 
$|AB|\ge|A|+|B|-r(g)$ for any element $g\in AB$. 
\end{theorem}

Quoting from~\refb{o1},
\begin{quotation}
``Theorem~\reft{KW} goes back to results of L.~Moser and P.~Scherk in the 
case  of abelian groups, and was proved for nonabelian groups by 
J.~H.~B.~Kemperman and (independently) D.~F.~Wehn. For proof see Kemperman's 
paper~\refb{k}.'' 
\end{quotation}

\section{Proof of Theorem~\reft{main}: sufficiency}\label{s:suff}

If $G$ is a group, $H$ is a finite subgroup of $G$, and $A$ is a subset of  
$G$ contained in an $H$-coset and satisfying $|A|>\frac35\,|H|$, then  
$A^{-1}A=H$ by the box principle, whence $|A^{-1}A|<\frac53\,|A|$. Thus, 
condition (i) of the theorem is sufficient for $A$ to satisfy  
$|A^{-1}A|<\frac53\,|A|$, and it also implies the corresponding part of the 
last assertion of the theorem. We now prove a similar result for 
condition~(ii). 

\begin{proposition}\label{p:NH}
Let $H$ be a finite subgroup of a group $G$, and suppose that $A\seq aH\cup 
bH$  where $a,b\in G$ are elements with $a^{-1}b\in N(H)$ and 
$(a^{-1}b)^2\notin H$. If $|A|>\frac95\,|H|$, then $|A^{-1}A|<\frac53\,|A|$; 
moreover, in this case $A^{-1}A$ is a disjoint union of $H$ and  two double 
$H$-cosets of size $|H|$ each. 
\end{proposition}

\begin{proof}
The assumption $|A|>\frac95\,|H|$ implies that the cosets $aH$ and $bH$ are 
disjoint. We write $A=aX\cup bY$ with $X,Y\seq H$ and notice that 
\begin{multline}\label{e:1902a}
  A^{-1}A = (X^{-1}a^{-1}\cup Y^{-1}b^{-1})(aX\cup bY) \\
		= ((X^{-1}X)\cup(Y^{-1}Y)) \cup(X^{-1}a^{-1}bY)\cup(Y^{-1}b^{-1}aX).
\end{multline}
Since $|X|+|Y|=|A|>|H|$, we have either $|X|>\frac12\,|H|$, or 
$|Y|>\frac12\,|H|$;  accordingly, by the box principle, either $X^{-1}X=H$, 
or $Y^{-1}Y=H$. Thus, $(X^{-1}X)\cup(Y^{-1}Y)=H$. Furthermore, since 
$a^{-1}b\in N(H)$, we have $a^{-1}bY\seq a^{-1}bH=Ha^{-1}b$; consequently, 
there is a subset $Y'\seq H$ such that $a^{-1}bY=Y'a^{-1}b$, and then 
$X^{-1}a^{-1}bY=X^{-1}Y'a^{-1}b=Ha^{-1}b$, as $X^{-1}Y'=H$ in view of 
$|X^{-1}|+|Y'|=|X|+|Y|>|H|$. Therefore $|X^{-1}a^{-1}bY|=|H|$. Taking the 
inverses, we get $|Y^{-1}b^{-1}aX|=|H|$. Hence, $|A^{-1}A|\le 
3|H|<\frac53\,|A|$. Next, $X^{-1}a^{-1}bY=Ha^{-1}b=a^{-1}bH$ shows that 
$X^{-1}a^{-1}bY$ is a double $H$-coset, and so is its inverse 
$Y^{-1}b^{-1}aX=Hb^{-1}a=b^{-1}aH$. Finally, the two double $H$-cosets are 
disjoint from $H$ and from each other thanks to the assumption 
$(a^{-1}b)^2\notin H$.
\end{proof}

Next, we prove a lemma that provides a simple criterion for a given set to  
satisfy conditions (i) and (ii) of Theorem~\reft{main}; this lemma will be 
used in the proof of necessity in the next section. 

\begin{lemma}\label{l:NH}
Let $H$ be a finite subgroup of a group $G$, and suppose that $A\seq aH\cup 
bH$ where $a,b\in G$. If $|A|>\frac95\,|H|$ and $|A^{-1}A|\le 3|H|$, then $A$ 
satisfies either condition (i), or condition (ii) of Theorem~\reft{main}, 
according to whether $(a^{-1}b)^2\in H$ or $(a^{-1}b)^2\notin H$. 
\end{lemma}

\begin{proof}
As in the proof of Proposition~\refp{NH}, we write $A=aX\cup bY$ with  
$X,Y\seq H$, and use~\refe{1902a}. If $a^{-1}b\notin N(H)$, then $Ha^{-1}b\ne 
a^{-1}bH$ by the definition of the normalizer subgroup; hence, 
$$ |X^{-1}a^{-1}b\cap a^{-1}bY| \le |Ha^{-1}b\cap a^{-1}bH|\le\frac12\,|H| $$
by Lemma~\refl{H2}. Without loss of generality, we assume $a,b\in A$  whence 
$1\in X\cap Y$. Consequently, both $X^{-1}a^{-1}b$ and $a^{-1}bY$ lie in 
$X^{-1}a^{-1}bY$, and we conclude that 
$$ |X^{-1}a^{-1}bY| \ge |X^{-1}a^{-1}b\cup a^{-1}bY|
= |X|+|Y|-|X^{-1}a^{-1}b\cap a^{-1}bY| > |H|. $$ Taking the inverses, we get 
$|Y^{-1}b^{-1}aX|>|H|$, and then from~\refe{1902a} we obtain 
$|A^{-1}A|>3|H|$, contradicting the assumptions. Thus, $a^{-1}b\in N(H)$. By 
Lemma~\refl{sub}, the set $F:=H\cup(a^{-1}b)H$ is a subgroup if and only if 
$(a^{-1}b)^2\in H$. In this case $A\seq aF$ and 
$|A|>\frac95\,|H|=\frac9{10}\,|F|>\frac35\,|F|$ so that $A$ satisfies 
condition (i). Finally, if $(a^{-1}b)^2\notin H$, then $A$ satisfies 
condition (ii). 
\end{proof}

\section{Proof of Theorem~\reft{main}: necessity}\label{s:proof}

Let $A\seq G$ be a finite subset with $|A^{-1}A|<\frac53\,|A|$, and suppose  
that the assertion is true for all sets $\cA\seq G$ satisfying either 
$|\cA^{-1}\cA|<|A^{-1}A|$, or $|\cA^{-1}\cA|=|A^{-1}A|$ and $|\cA|>|A|$. We 
show that $A$ is contained ether in a coset, or in a union of two cosets, as 
specified in the conditions (i) and (ii) of the theorem. 

We write $Q:=A^{-1}A$; thus, $|A|>\frac35\,|Q|$. 

Recall, that for an element $g\in G$, we have denoted by $r(g)$ the number of  
representations $g=a^{-1}b$ with $a,b\in A$. By Theorem~\reft{KW}, 
\begin{equation*}\label{e:rgQ}
	r(g) \ge 2|A| - |Q|, \quad g\in Q.
\end{equation*}

Let $Q^+:=\{g\in Q\colon r(g)>|Q|-|A|\}$. We notice that $Q^+$ is nonempty 
as, for instance, it contains the identity element. Also, $Q^+$ is stable 
under inversion. 

For any $g\in Q$ and $g_0\in Q^+$ we have 
  $$ r(g)+r(g_0)>(2|A|-|Q|)+(|Q|-|A|) = |A|.$$ 
Hence, $g_0^{-1}g\in Q$ by Lemma~\refl{box}, implying $g\in g_0Q$. It follows 
that $g_0Q=Q$ for any $g_0\in Q^+$. Therefore, $Q^+Q=Q$ and, considering the 
inverses, $QQ^+=Q$. Letting $F:=\<Q^+\>$, we furthermore conclude that 
$QF=FQ=Q$. As a result, 
\begin{equation}\label{e:1002a}
  (AF)^{-1}(AF) = FA^{-1}AF = FQF = FQ = Q = A^{-1}A.
\end{equation} 

From these equalities and by the choice of $A$, either $AF=A$, or there is a 
finite subgroup $H\le G$ such that one of the following holds: 
\begin{itemize}
\item[--] $AF$ is contained in a left $H$-coset, $|AF|>\frac35\,|H|$, and    
	$(AF)^{-1}(AF)=H$;
\item[--] $AF$ meets exactly two left $H$-cosets, $|AF|>\frac95\,|H|$, and  
    $(AF)^{-1}(AF)$ is a disjoint union of $H$ and two double $H$-cosets of 
    size $|H|$ each. 
\end{itemize} 
In the first case, recalling~\refe{1002a} we get $Q=(AF)^{-1}(AF)=H$; as a 
result, $A$ is contained in a single left $H$-coset, and 
$|A|>\frac35\,|Q|=\frac35\,|H|$; thus, $A$ satisfies condition (i). 

In the second case, with \refe{1002a} in mind, $|Q|=|(AF)^{-1}(AF)|=3|H|$  
showing that $|A|>\frac35\,|Q|=\frac95\,|H|$. By Lemma~\refl{NH}, the set $A$ 
satisfies condition (i) or condition (ii). 

Having ruled out the exceptional cases where $AF$ is contained in a single  
coset, or in a union of two cosets, we proceed with the proof using the 
additional assumption 
\begin{equation}\label{e:AF}
  AF=A, \quad F=\<Q^+\>. 
\end{equation}  
Thus, $A$ is a union of left $F$-cosets, and so is $Q=\cup_{a\in A}a^{-1}A$. 
It follows that $F\seq Q$. Indeed, we have $F=Q^+$: here the inclusion 
$Q^+\seq F$ is trivial, while $F\seq Q^+$ follows by observing that if $g\in 
F$, then $r(g)=|A|>|Q|-|A|$ by~\refe{AF}, whence $g\in Q^+$. 

As we have just observed, if $g\in F$, then $r(g)=|A|$. Conversely, if $g\in 
Q$ is an element with $r(g)=|A|$, then $r(g)>|Q|-|A|$ showing that $g\in 
Q^+=F$. As a bottom line, $r(g)=|A|$ if and only if $g\in F$.

If $g\in G$ is an element with $r(g)>|Q|-|A|$, then $g\in Q^+=F$, whence, 
indeed, $r(g)=|A|$. Thus, we have 
  $$ 2|A|-|Q| \le r(g) \le |Q|-|A| $$ 
for all elements $g\in Q$ with $r(g)<|A|$. We remark that the first of the 
two inequalities is just Theorem~\reft{KW}, but the second one is new and, in 
our view, is quite amazing. 

We write $k:=|A|/|F|$. If $k=1$, then $A$ is a single left $F$-coset; 
therefore $A$ satisfies condition (i). If $k=2$, then $A$ is a union of two 
left $F$-cosets and $|A^{-1}A|<\frac53|A|<4|F|$; therefore, applying 
Lemma~\refl{NH}, we conclude that $A$ satisfies condition (ii). Suppose thus 
that $k\ge 3$. 

Since $A$ and $Q$ are unions of left $F$-cosets, both $|A|$ and $|Q|$ are 
divisible by $|F|$. From this observation and 
$|Q|<\frac53\,|A|=\frac{5}3\,k|F|$, we get 
\begin{equation}\label{e:QF}
  |Q|\le\lpr\lcl\frac{5}{3}\,k\rcl-1\rpr\,|F|.
\end{equation} 
Furthermore, to any representation $g=a^{-1}b$ with $a,b\in A$ there 
correspond $|F|$ representations $g=(fa)^{-1}(fb),\ f\in F$. (Notice that 
these representations are ``legal'' in the sense that both $af$ and $bf$ lie 
in $A$.) Therefore, also $r(g)$ is divisible by $|F|$, for any $g\in G$. 
Moreover, since $g$ is constant on any left $F$-coset, for any given positive 
integer $m$, the number of elements $g\in Q$ with $r(g)=m$ is  divisible by 
$|F|$;  

Let $Q_0$ and $Q_1$ denote the sets of all those elements $g\in Q$ with
$r(g)\le\frac12\,|A|$ and with $r(g)>\frac12\,|A|$, respectively. We write
$N_0:=|Q_0|$ and $N_1:=|Q_1|$ and define
  $$ \sig_0 := \sum_{g\in Q_0}r(g)
                              \ \text{and}\ \sig_1:=\sum_{g\in Q_1}r(g); $$
thus, $N_0+N_1=|Q|$, $\sig_0+\sig_1=|A|^2$, and $N_0,N_1,\sig_0$, and 
$\sig_1$  are all divisible by $|F|$. The sum $\sig_0$ has $N_0$ terms, each 
of them divisible by $|F|$ and not exceeding $\frac12\,|A|=\frac12\,k|F|$; 
therefore, $\sig_0\le\lfl\frac12\,k\rfl|F|N_0$. The sum $\sig_1$ has $N_1$ 
terms, of them $|F|$ are equal to $|A|$, and each of the remaining $N_1-|F|$ 
terms does not exceed $|Q|-|A|$. 
Therefore,
  $$ \sig_1 \le |F||A| + (N_1-|F|)(|Q|-|A|) 
                                  = 2|F||A|-|Q||F| + N_1|Q| - N_1|A|. $$
Letting $n:=N_1/|F|$ and $q:=|Q|/|F|$, we obtain
\begin{align*}
  |A|^2  &\le \lfl\frac 12\,k\rfl|F|N_0 
                                  + 2|F||A|-|Q||F| + N_1|Q| - N_1|A|, \\
  k^2|F| &\le \lfl\frac12\,k\rfl (N_0+N_1) +2|A| - |Q| 
                 + \left(q-\lfl\frac12k\rfl\right)N_1 - kN_1, \\
  k^2    &\le \lpr\lfl\frac12\,k\rfl -1\rpr\,q + 2k
                            + \left(q-k-\lfl\frac12\,k\rfl\right)n.
\end{align*}
Since $q\le\lcl\frac53\,k\rcl-1$ by~\refe{QF}, we derive that
  $$ k^2 \le \lpr \lfl\frac12\,k\rfl -1 \rpr 
     \lpr \lcl\frac53\,k\rcl -1 \rpr + 2k 
        + \lpr\lcl\frac53\,k\rcl-1 -  k - \lfl  \frac12\,k\rfl  \rpr n. $$
A routine analysis shows that for $k\ge 3$, the last inequality is false if 
$n\le k$. (Hint: exact computation for $3\le k\le 6$, substituting $n=k$ and 
using the crude estimates $\lfl k/2\rfl\le k/2$ and $\lcl 
5k/3\rcl\le(5k+2)/3$  for $k\ge 7$.) Therefore $n\ge k+1$; that is, 
\begin{equation}\label{e:01A}
  |Q_1| = N_1 \ge |A|+|F| \quad (k\ge 3).
\end{equation} 

From Lemma~\refl{box} and the definition of the set $Q_1$, we have 
$g_1^{-1}g_2\in Q$ for any $g_1,g_2\in Q_1$. Consequently, $Q_1^{-1}Q_1\seq 
Q$ whence, by the choice of $A$, there is a finite subgroup $H\le G$ such 
that one of the following holds: 
\begin{itemize}
\item[1.\,] $Q_1$ is contained in a left $H$-coset and
    $|Q_1|>\frac35\,|H|$;
\item[2.\,] $Q_1$ meets exactly two left $H$-cosets and  
    $|Q_1|>\frac95\,|H|$.    
\end{itemize} 

We investigate these two cases separately.

\smallskip\noindent{\bf Case 1:}\ There is a finite subgroup $H\le G$ such  
that $Q_1$ is contained in a left $H$-coset and $|Q_1|>\frac35\,|H|$. We have   
\begin{equation}\label{e:HQ1}
  |A| < |Q_1|\ \text{and}\ Q_1\seq H = Q_1^{-1}Q_1 \seq Q;
\end{equation}
here the inequality follows from~\refe{01A}, the first inclusion from $1\in 
Q_1$, the equality from  $|Q_1|>\frac35\,|H|$ and the box principle, and the 
second inclusion from Lemma~\refl{box}. 

Consider the coset decomposition $A=A_1\longu A_n$ where $n=|AH|/|H|$ and 
$A_1\longc A_n$ are nonempty and reside in pairwise distinct left $H$-cosets. 
We number the sets $A_i$ so that $|A_1|=\min\{|A_i|\colon 1\le i\le n\}$. Fix 
$a_1\in A_1$. In view of $a_1^{-1}A_2\longu a_1^{-1}A_n\seq Q\stm H$ 
and~\refe{HQ1}, 
  $$ |Q| \ge |H| + (|A_2|\longp|A_n|) 
        > |A|+\lpr1-\frac1n\rpr\,|A| = \lpr2-\frac1n\rpr\,|A|. $$ 
Since $|Q|<\frac53\,|A|$, we conclude that $n=1$ or $n=2$. If $n=1$, then $A$ 
resides in a single $H$-coset; moreover, $|A|>\frac35\,|Q|>\frac35\,|H|$, 
showing that $A$ satisfies condition (i). 

Suppose now that $n=2$. Fix $a_1\in A_1$ and $a_2\in A_2$. Then 
  $$ \frac53\,|A| > |Q| = |H| + |Q\stm H| \ge |H| + |a_1^{-1}A_2| 
           = |H| + (|A|-|A_1|)  $$
whence $|A_1|>|H|-\frac23\,|A| \ge \frac13\,|A|$. Similarly, from
  $$ \frac53\,|A| > |Q| = |H| + |Q\stm H| \ge |H| + |a_1^{-1}A_2| 
           = |H| + |A_2|  $$
we obtain $|A_2|<\frac23\,|A|$. Therefore, 
\begin{equation}\label{e:halfH}
  \frac13\,|H| <|A_1|\le\frac12\,|A|\le|A_2|<\frac23\,|A|.           
\end{equation}
Consider the set $S:=(A_1\times A_2)\cup(A_2\times A_1)$ and the mapping 
$\phi\colon S\to Q$ defined by $\phi(a,b):=a^{-1}b$. Since the image 
$\Im(\phi)$ is disjoint from $H$, we have $\Im(\phi)\seq Q\stm Q_1$. By the 
definition of the set $Q_1$, every element of $\Im(\phi)$ has at most 
$\frac12\,|A|$ inverse images in $S$. As a result, 
  $$ |Q|-|Q_1|\ge\Im(\phi) 
                      \ge \frac{|S|}{|A|/2} = 4\frac{|A_1||A_2|}{|A|}. $$ 
Comparing this estimate with~\refe{01A} and with the assumption 
$|Q|<\frac53\,|A|$, we obtain $|A|^2>6|A_1||A_2|$. This leads to 
$|A_2|>(2+\sqrt 3)|A_1|$, contradicting~\refe{halfH}.  

\smallskip\noindent{\bf Case 2:}\ There is a finite subgroup $H\le G$ such 
that $Q_1$ meets exactly two left $H$-cosets and $|Q_1|>\frac95\,|H|$. Since 
$1\in Q_1$, we can write $Q_1=B_0\cup B_1$ where $B_0\seq H$ and $B_1\seq gH$ 
with some $g\in G\stm H$. From $B_1=Q_1\stm H$ and $Q_1^{-1}=Q_1$ we get 
$B_1^{-1}=Q_1^{-1}\stm H=Q_1\stm H=B_1\seq gH$. Thus, $B_1\seq gH\cap 
Hg^{-1}$. By Lemma~\refl{leftright}, and in view of 
  $$ |B_1|=|Q_1|-|B_0|\ge|Q_1|-|H|>\frac45\,|H| $$
we have $gH=Hg^{-1}$; that is, $H=gHg$, and it is easily seen that $H\cup gH$ 
is a subgroup. Moreover, $Q_1=B_0\cup B_1\seq H\cup gH$ and 
$|Q_1|>\frac95\,|H|>\frac35\,|H\cup gH|$. This takes us back to the Case~1 
considered above. 

\section{Concluding remarks}\label{s:conjecture}

What is the structure of a finite set $A$ with $\frac53\,|A|<|A^{-1}A|<2|A|$? 
We make the following conjecture. 
\begin{conjecture}\label{j:main}
Let $A$ be a finite subset of a group $G$, and let $n$ be a positive integer. 
If
  $$ 
                         |A^{-1}A| < \lpr2-\frac1{n+1}\rpr\,|A|, $$ 
then there are a finite subgroup $H\le G$ and a subset $A_0\seq A$ of size 
$|A_0|\le n$ contained in a single left $N(H)$-coset such that $A\seq A_0H$, 
$|A_0H|=|A_0||H|$, and $|A|>\lpr2-\frac1{n+1}\rpr^{-1}(2|A_0|-1)|H|$. 

Moreover, $A^{-1}A=A_0^{-1}A_0H$ and $|A^{-1}A|=(2|A_0|-1)H$.
\end{conjecture}

The inequality $|A|>\lpr2-\frac1{n+1}\rpr^{-1}(2|A_0|-1)|H|$ is worth 
commenting on. It can be shown that, along with other conclusions of the 
conjecture, it implies $|A|\le |A_0H| < |A|+ \frac{n}{2n+1}\,|H|$. Thus, this 
inequality ensures that $A$ is a \emph{dense} subset of the set $A_0H$. 

The particular case $n=1$ of the conjecture follows from Olson's theorem, 
while the case $n=2$ is the main result of this paper. 

As the following proposition shows, in the appropriate range,  
Conjecture~\refj{main} gives a necessary and sufficient condition for $A$ to 
satisfy $|A^{-1}A| < \lpr2-\frac1{n+1}\rpr\,|A|$. 
\begin{proposition}\label{p:sharp}
Let $A$ be a finite subset of a group $G$, and let $n$ be a positive integer. 
Suppose that there are a finite subgroup $H\le G$ and a subset $A_0\seq A$ of 
size $|A_0|\le n$ contained in a single left $N(H)$-coset such that $A\seq 
A_0H$, $|A_0H|=|A_0||H|$, and  $|A|>\lpr2-\frac1{n+1}\rpr^{-1}(2|A_0|-1)|H|$. 
If, in addition, $|A^{-1}A|<2|A|$ then, indeed, 
$|A^{-1}A|<\lpr2-\frac1{n+1}\rpr\,|A|$. 
\end{proposition}

We omit the proof since, anyway, Proposition~\refp{sharp} is not of much 
importance as long as Conjecture~\refj{main} remains open.

\bigskip
\end{document}